\documentclass[10pt,a4paper]{amsart}
\usepackage{amssymb,amsmath,latexsym}

\newtheorem{theorem}{Theorem}[section]

\newtheorem{definition}[theorem]{Definition}

\newtheorem{lemma}[theorem]{Lemma}

\newtheorem{remark}[theorem]{Remark}

\newcommand{\nN}{\mathbb{N}}

\title[Continuity of modulus of noncompact convexity]{Continuity of modulus of noncompact convexity for  minimalizable measures of noncompactness}
\author[A. Reki\' c-Vukovi\' c]{Amra Reki\' c-Vukovi\' c}
\email{amra.rekic@untz.ba}
\author[N. Oki\v ci\' c]{Nermin Oki\v ci\' c}
\email{nermin.okicic@untz.ba}
\author[V. Pa\v si\' c]{Vedad Pa\v si\' c}
\email{vedad.pasic@untz.ba}
\author[I. Arandjelovi\' c]{Ivan Arandjelovi\' c}
\email{iarandjelovic@mas.bg.ac.rs}
\address{Department of Mathematics
University of Tuzla
Univerzitetska 4
75000 Tuzla
Bosnia and Herzegovina}

\address{Faculty of Mechanical Engineering
University of Belgrade
Kraljice Marije 16
11000 Belgrade
Serbia}

\begin{document}


\newcommand{\AuthorNames}{A. Reki\' c-Vukovi\' c et al.}

\newcommand{\FilMSC}{46B20; 46B22}
\newcommand{\FilKeywords}{modulus of noncompact convexity, measure of noncompactness}
\newcommand{\FilCommunicated}{(name of the Editor, mandatory)}

\maketitle

\begin{abstract}
We consider the modulus of noncompact convexity
$\Delta_{X,\phi}(\varepsilon)$ associated with the minimalizable measure of noncompactness $\phi$. We present some properties of this modulus, while the main result of this paper is showing that $\Delta _{X,\phi }(\varepsilon)$ is a subhomogenous and continuous function on $[0,\phi (\overline{B}_X))$ for an arbitrary minimalizable measure of compactness $\phi$ in the case of a Banach space $X$ with the Radon-Nikodym property.
\end{abstract}

\makeatletter
\renewcommand\@makefnmark%
{\mbox{\textsuperscript{\normalfont\@thefnmark)}}}
\makeatother

\section{Introduction}\label{introSec}
One of the basic concepts of geometry of Banach spaces is that of uniform convexity, introduced by Clarkson \cite{article.3}.  Smulian \cite{article.11} characterized the property dual to uniform convexity, i.e. uniform smoothness  and pointed out that uniform convexity and uniform smoothness describe geometric properties of finite dimensional subspaces of a normed space. Most of the problems in fixed point theory, however, have global character, which was the motivation for considering infinite dimensional counterparts of classical geometric notions. One of these is nearly uniform convexity introduced by Huff in \cite{article.6}.
 Goebel and Sekowski \cite{article.5} define the modulus corresponding to nearly uniform convexity and use the concept of measure of noncompactness to give a new classification of Banach spaces. They prove that whenever the characteristic of uniform noncompact convexity of any Banach space is less than 1, the space is reflexive and has normal structure.

\subsection{Fundamental concepts and definitions}

In this paper $X$ denotes the Banach space, $B(x,r)$ an open ball centred at $x$ of radius $r$, $B_X$ the unit ball and $S_X$ the unit sphere in $X$. If $A\subset X$ we denote by $\overline{A}$ the closure of a set $A$ and by $coA$ the convex envelope of $A$.

\begin{definition}\label{measureDefn}
Let $\mathcal{B}$ be a family of bounded subsets of $X$. We call the mapping $\phi:\mathcal{B}\rightarrow$ $[0,+\infty)$ the \emph{measure of noncompactness} defined on $X$ if it satisfies the following :
\begin{enumerate}

\item $\phi (B)=0\Leftrightarrow B$ is a relatively compact set;

\item $\phi (B)=\phi (\overline{B})$, $\forall B\in\mathcal{B}$

\item $\phi (B_1\cup B_2)=\max \{ \phi (B_1),\phi (B_2)\}$, $\forall B_1,B_2\in\mathcal{B}$.

\end{enumerate}
Some of well known measures of noncompactness are Kuratowski measure ($\alpha$), Hausdorff measure ($\chi$) and Istratescu measure ($\beta$), see e.g. \cite{book2}, \cite{book1} and \cite{article.9}.
We call a measure of noncompactness $\phi$ a \emph{minimalizable measure of noncompactness} if for every infinite bounded set $A$ and for every $\varepsilon>0$ there exists its subset $B\subset A$ which is $\phi$-minimal and such that  $\phi(B)\geq \phi(A)-\varepsilon$. \\
A measure $\phi$ is a \emph{strictly minimalizable measure of noncompactness} if for every infinite, bounded set $A$ there exists its subset $B\subset A$ which is $\phi$-minimal and such that $\phi(B)=\phi(A)$.
\end{definition}
\begin{remark}
Clearly every strictly minimalizable measure is a minimalizable measure of noncompactness. See e.g. \cite{book1} and \cite{book2} for more on measures of noncompactness.
\end{remark}
\begin{definition}\label{modulusDefn}
A modulus of noncompact convexity associated with an arbitrary measure of noncompactness $\phi$ is a function $\Delta_{X,\phi}:[0,\phi(B_X)]\rightarrow [0,1]$ given by
\[\Delta_{X,\phi}(\varepsilon)=\inf \{ 1-d(0,A): A\subseteq \overline{B_X}, A=coA=\overline{A}, \phi(A)> \varepsilon\} \ .\]
\end{definition}
Banas \cite{article.1} considered a modulus $\Delta_{X,\phi}(\varepsilon)$ for $\phi=\chi $, where $\chi$ is the Hausdorff measure of noncompactness. \\ For $\phi =\alpha$, where $\alpha$ is a Kuratowski measure of noncompactness,   $\Delta_{X,\alpha}(\varepsilon)$ represents the Goebel-Sekowski modulus of noncompact convexity, see \cite{article.5}. \\ For $\phi =\beta$, where $\beta$ is the  separation measure of noncompactness, $\Delta_{X,\beta}(\varepsilon)$ represents the  Dominguez-Lopez modulus of noncompact convexity, see \cite{article.2}.
\begin{definition}\label{convexityDefn}
We define the characteristic of noncompact convexity of $X$ associated with a measure of noncompactness $\phi$ by
\[\varepsilon_{\phi}(X)=\sup\{ \varepsilon\geq 0: \Delta_{X,\phi}(\varepsilon)=0\} \ .\]
\end{definition}
For moduli $\Delta_{X,\phi}(\varepsilon)$ with respect to $\phi = \alpha, \chi, \beta$ we have the following inequalities:
\[\Delta_{X,\alpha}(\varepsilon)\leq \Delta_{X,\beta}(\varepsilon)\leq \Delta_{X,\chi}(\varepsilon) \ ,\]
and hence
\[\varepsilon _{\alpha} (X)\geq \varepsilon_{\beta} (X)\geq \varepsilon_{\chi}(X) \ .\]
It is known that $X$ is nearly uniformly convex (NUC) if and only if  $\varepsilon _{\phi }(X)=0$, for $\phi =\alpha, \chi, \beta$, see e.g. \cite{book2}.

Banas \cite{article.1} showed that the modulus $\Delta _{X,\chi }(\varepsilon)$ in the case of a reflexive space $X$ is a subhomogenous function continuous on $[0,1)$. Prus \cite{article.8} gave a result which links the continuity of the modulus $\Delta_{X,\phi}(\varepsilon)$ with the uniform Opial condition which implies the normal structure of the space. In \cite{article.10} we  demonstrated certain properties of the modulus $\Delta_{X,\phi}(\varepsilon)$ as well as its continuity for a strictly minimalizable measure of noncompactness $\phi$ on Banach spaces with the Radon-Nikodym property.
\begin{definition}\label{RNDefn}
We say that the Banach space $X$ is with the Radon-Nikodym property if and only if every nonempty bounded subset $A\subset X$ is dentable, i.e. if and only if for all $\varepsilon >0$, exists $x\in A$, such that $x\notin \overline{co}(A\backslash \overline{B}(x,\varepsilon))$, see e.g. \cite{article.4}.
\end{definition}

The aim of this paper is to show that the modulus $\Delta _{X,\phi }(\varepsilon)$ is a subhomogenous and continuous function on $[0,\phi (B_X))$ for an arbitrary minimalizable measure of compactness $\phi$ in the case of a Banach space $X$ with the Radon-Nikodym property.

{\bf{Structure of paper.}}
This paper has the following structure. Section \ref{introSec} provides the background information and the fundamental concepts and definitions. Section \ref{propertiesSec} contains several results which provide the properties of the $\Delta_{X,\phi}$ modulus. Section \ref{contSec} contains the main result of the paper, namely the result on the continuity of the modulus $\Delta_{X,\phi}$. Section \ref{discussionSec} contains the discussion of the obtained results.

\section{$\Delta_{X,\phi}$ modulus properties}\label{propertiesSec}

\begin{theorem}{\label{TeoremSubhomogenost}}
Let $X$ be a Banach space with Radon-Nikodym property and let $\phi$ be a minimalizable measure of noncompactness. The modulus $\Delta_{X,\phi}(\varepsilon)$ is a subhomogenous function, i.e. we have that for all $k\in [0,1]$ and for all $\varepsilon \in [0,\phi (\overline{B}_X)]$
\[
\Delta_{X,\phi}(k\varepsilon)\leq k\Delta_{X,\phi}(\varepsilon) \ .
\]
\end{theorem}
\begin{proof}
Let $\eta >0$ be arbitrary and  $\varepsilon \in [0,\phi (\overline{B}_X)]$. By definition \ref{modulusDefn} there exists a convex closed subset $A\subset \overline{B}_X$, $\phi(A)> \varepsilon$ such that
\[1-d(0,A)< \Delta_{X,\phi}(\varepsilon)+\eta \ . \]
For arbitrary $k\in [0,1]$ the set $kA\subset \overline{B}_X$ is convex and closed. Hence we have that $\phi (kA)=k\phi (A)> k\varepsilon$.

As $\phi$ is a minimalizable measure, for $\displaystyle \delta =\frac{\phi (kA)-k\varepsilon}{2}>0$ there exists an infinite $\phi$-minimal set $B\subset kA$, such that
\[\phi (B)\geq \phi (kA)-\delta \ ,\]
i.e. $\phi (B)> k\varepsilon$. From the Radon-Nikodym property of the space $X$, for $\varepsilon >0$ there exists $x_0\in B$ such that
\[x_0\notin \overline{co}\left [B\backslash \overline{B}\left (x_0, \frac{\varepsilon }{2}\right )\right ] \ .\]
If we define the set $\displaystyle B^*=C+\frac{1-k}{\|x_0\|}x_0$, where $\displaystyle C=\overline{co}\left [B\backslash \overline{B}\left (x_0, \frac{\varepsilon }{2}\right )\right ]$, then
\begin{equation}{\label{th2.2}}
1-d(0,B^*)<k(\Delta _{X,\phi }(\varepsilon)+\eta) \ .
\end{equation}
The set $B^*\subset \overline{B}_X$ is closed and convex and we have that $\phi (B^*)=\phi (C)=\phi (B)> k\varepsilon$.
If we take the infimum over all the sets $B^*$, $\phi (B^*)> k\varepsilon$, in inequality (\ref{th2.2}) we get that
\[\Delta_{X,\phi}(k\varepsilon)\leq k(\Delta_{X,\phi}(\varepsilon)+\eta) \ ,\]
which proves the theorem due to $\eta >0$ being arbitrary.
\end{proof}

Several additional properties of the modulus $\Delta_{X,\phi}$  arise from Theorem \ref{TeoremSubhomogenost}.

\begin{lemma}
Let $\phi$ be a minimalizable measure of noncompactness defined on a Banach space $X$ with the Radon-Nikodym property. Then the modulus  $\Delta_{X,\phi}(\varepsilon)$ is a strictly increasing function on $[\varepsilon_{\phi}(X),\phi(\overline{B}_X)]$.
\end{lemma}
\begin{proof}
Let $\varepsilon_1$, $\varepsilon_2\in [\varepsilon_{\phi}(X),\phi(\overline{B}_X)]$ and $\varepsilon_1<\varepsilon_2$. If we put $\displaystyle k=\frac{\varepsilon_1}{\varepsilon_2}<1$, using Theorem \ref{TeoremSubhomogenost} we have that
\[\Delta_{X,\phi}(\varepsilon_1)=\Delta_{X,\phi}(k\varepsilon_2)\leq k\Delta_{X,\phi}(\varepsilon_2)<\Delta_{X,\phi}(\varepsilon_2) \ .\]
\end{proof}
\begin{lemma}
Let $\phi$ be a minimalizable measure of noncompactness defined on a Banach space $X$ with the Radon-Nikodym property. For every $\varepsilon \in [0,\phi(\overline{B}_X)]$ we have that
\[\Delta_{X,\phi}(\varepsilon)\leq \varepsilon \ .\]
\end{lemma}
\begin{proof}
If $\varepsilon \in [0,1]$, then using Theorem \ref{TeoremSubhomogenost}, switching the roles of $k$ and $\varepsilon$ around and putting $\varepsilon =1$, we have that
\[\Delta_{X,\phi}(\varepsilon)\leq \varepsilon \Delta_{X,\phi}(1)\leq \varepsilon \ .\]
If $\varepsilon \in (1,\phi(\overline{B}_X)]$, then due to monotonicity of the modulus $\Delta_{X,\phi}(\varepsilon)$ we have that
\[\Delta_{X,\phi}(\varepsilon)< \Delta_{X,\phi}(\phi(\overline{B}_X))=1<\varepsilon \ .\]
\end{proof}
\begin{lemma}
Let $\phi$ be a minimalizable measure of noncompactness defined on a Banach space $X$ with the Radon-Nikodym property. Then the function $\displaystyle f(\varepsilon)=\frac{\Delta_{X,\phi}(\varepsilon)}{\varepsilon}$ is nondecreasing on $[0,\phi(\overline{B}_X)]$ and for $\varepsilon_1+\varepsilon_2\leq \phi(\overline{B}_X)$ we have that
\begin{equation}{\label{poslj3nejednakost}}
\Delta_{X,\phi}(\varepsilon_1+\varepsilon_2)\geq \Delta_{X,\phi}(\varepsilon_1)+\Delta_{X,\phi}(\varepsilon_2) \ .
\end{equation}
\end{lemma}
\begin{proof}
Let $\varepsilon_1$, $\varepsilon_2\in [0,\phi (\overline{B}_X)]$ be such that $\varepsilon_1\leq \varepsilon_2$. Putting  $\displaystyle k=\frac{\varepsilon_1}{\varepsilon_2}$, we get that
\[f(\varepsilon_1)=\frac{\Delta_{X,\phi}(\varepsilon_1)}{\varepsilon_1}= \frac{\Delta_{X,\phi}(k\varepsilon_2)}{\varepsilon_1} \ .\]
Using the subhomogeneity of the function $\Delta_{X,\phi}(\varepsilon)$ we have that
\[f(\varepsilon_1)\leq \frac{\Delta_{X,\phi}(\varepsilon_2)}{\varepsilon_2}=f(\varepsilon_2) \ ,\]
which shows that $f(\varepsilon)$ is a nondecreasing function on $[0,\phi(B_X)]$. We also have that
\begin{eqnarray*}
\Delta_{X,\phi}(\varepsilon_1)+\Delta_{X,\phi}(\varepsilon_2) & \leq & k\Delta_{X,\phi}(\varepsilon_2)+\Delta_{X,\phi}(\varepsilon_2) \\
 & = & \frac{\varepsilon _1+\varepsilon_2}{\varepsilon_2}\Delta_{X,\phi}(\varepsilon_2) \\
 & \leq & (\varepsilon _1+\varepsilon_2)\frac{\Delta_{X,\phi}(\varepsilon_1+\varepsilon _2)}{\varepsilon_1+\varepsilon_2}\\
 & = & \Delta_{X,\phi}(\varepsilon_1+\varepsilon _2) \ ,
\end{eqnarray*}
which shows inequality (\ref{poslj3nejednakost}).
\end{proof}
\begin{lemma}
Let $\phi$ be a minimalizable measure of noncompactness on a Banach space $X$ with the Radon-Nikodym property. Then for all $\varepsilon_1,\varepsilon_2 \in(\varepsilon_1(X),$ $\phi(\overline{B}_X)]$,  such that $\varepsilon_1\leq \varepsilon_2$, we have that
\begin{equation}{\label{poslj4nejednakost}}
\frac{\Delta_{X,\phi}(\varepsilon_2)-\Delta_{X,\phi}(\varepsilon_1)}{\varepsilon_2-\varepsilon_1}\geq
\frac{\Delta_{X,\phi}(\varepsilon_2)}{\varepsilon_2} \ .
\end{equation}
\end{lemma}
\begin{proof}
Let $\displaystyle k=\frac{\varepsilon_1}{\varepsilon_2}\leq 1$. Using Theorem \ref{TeoremSubhomogenost} we get that
\begin{eqnarray*}
\Delta_{X,\phi}(\varepsilon_2)-\Delta_{X,\phi}(\varepsilon_1) & = &
\Delta_{X,\phi}(\varepsilon_2)-\Delta_{X,\phi}(k\varepsilon_2) \\
& \geq & \Delta_{X,\phi}(\varepsilon_2)-k\Delta_{X,\phi}(\varepsilon_2) \\
& = & \frac{\varepsilon_2-\varepsilon_1}{\varepsilon_2}\Delta_{X,\phi}(\varepsilon_2),
\end{eqnarray*}
which shows the inequality (\ref{poslj4nejednakost}).
\end{proof}

\section{The continuity of the modulus of noncompact convexity}\label{contSec}
The main result of this paper is the following
\begin{theorem}\label{mainThm}
Let $\phi$ be a minimalizable measure of noncompactness and let $X$ be a Banach space with the Radon-Nikodym property. Then the modulus of noncompact convexity $\Delta_{X,\phi}(\varepsilon)$ is a continuous function on  $[0,\phi (\overline{B}_X))$.
\end{theorem}
The proof of Theorem \ref{mainThm} follows from two lemmas that deal with the continuity of the modulus $\Delta_{X,\phi}(\varepsilon)$ from below and above respectively, i.e. Lemma \ref{contBelow} and Lemma \ref{contAbove}.
\begin{lemma}\label{contBelow}
Let $\phi$ be a minimalizable measure of compactness and let $X$ be a Banach space with the Radon-Nikodym property. Then the modulus of noncompact convexity $\Delta_{X,\phi}(\varepsilon)$ is a continuous function from below on $[0,\phi (\overline{B}_X))$.
\end{lemma}
\begin{proof}
Let $\varepsilon_0\in [0,\phi (\overline{B}_X))$ be arbitrary and let $\varepsilon <\varepsilon_0$. By Definition \ref{modulusDefn}, for arbitrary $\eta >0$ there exists a convex and closed subset $A\subset \overline{B}_X$, $\phi (A)> \varepsilon$ such that
\[1-d(0,A)<\Delta_{X,\phi}(\varepsilon)+\eta \ .\]
If $ \phi (A)\geq \varepsilon_0>\varepsilon$, then clearly
\[\inf \{ 1-d(0,A): A=\overline{co}A, A\subset \overline{B}_X, \phi (A)\geq \varepsilon _0\} \leq \Delta_{X,\phi}(\varepsilon)+\eta \ ,\]
i.e.
\[\Delta_{X,\phi}(\varepsilon_0)\leq \Delta_{X,\phi}(\varepsilon)+\eta \ ,\]
 whence, due to $\eta >0$ being arbitrary, we complete the proof. Therefore, let $\varepsilon < \phi (A)<\varepsilon_0$ and $\gamma
=\varepsilon _0- \phi (A)>0$. Since the measure of noncompactness  $\phi$ is minimalizable,  for arbitrary $\gamma >0$, there exists an infinite $\phi$-minimal subset $B\subset A$ such that
\[\phi (B)\geq \phi (A)-\gamma =2\phi (A)-\varepsilon_0 \ .\]
Let $n_0\in \mathbb{N}$ be such that $\displaystyle \frac{2\phi (A)-\varepsilon_0}{n_0}<\frac{diam B}{2}$. Since $B\subset \overline{B}_X$ is a bounded subset of the space $X$ with the Radon-Nikodym property, for $\displaystyle r=\frac{2\phi (A)-\varepsilon_0}{n_0}$ there exists $x_0\in B$ such that
\[x_0\notin \overline{co}\left (B\backslash \overline{B}(x_0,r)\right ).\]
Let $C=\overline{co}\left (B\backslash \overline{B}(x_0,r)\right )$. The set $C\subset B$ is convex and closed and we have that $1-d(0,C)\leq 1-d(0,B)\leq 1-d(0,A)<\Delta_{X,\phi}(\varepsilon)+\eta .$ Besides, since $B$ is a $\phi$-minimal set, we have that
\[\phi (C)=\phi (B)\geq 2\phi (A)-\varepsilon _0 \ .\]
Let $\displaystyle k=1+\frac{1-d(0,C)}{2}$ and let us look at the set $\displaystyle A^*=kC \cap \overline{B}_X$. $A^*$ is closed and convex and we have that $A^*\subseteq kC\subset kB$ and hence
\[1-d(0,A^*)\leq 1-d(0,kC)<1-d(0,C)<\Delta_{X,\phi}(\varepsilon)+\eta \ .\]
Since $B$ is $\phi$-minimal so is $kB$ and hence
\[\phi (A^*)=\phi (kB)=k\phi (B)\geq k(2\phi (A)-\varepsilon_0) \ .\]
Put $\displaystyle \delta =\frac{\varepsilon_0}{2}\left(1-\frac{1}{k}\right )$. Then for arbitrary $\varepsilon \in (\varepsilon_0-\delta, \varepsilon_0)$ we have that
\[\phi (A^*)\geq k(2\varepsilon-\varepsilon _0)>k(2(\varepsilon_0-\delta)-\varepsilon _0)=\varepsilon_0 \ .\]
Therefore
\[\inf \{ 1-d(0,A^*): A^*\subset \overline{B}_X, A^*=\overline{co}A^*, \phi (A^*)>\varepsilon _0\} \leq \Delta_{X,\phi}(\varepsilon)+\eta \]
and hence
\[\Delta_{X,\phi}(\varepsilon_0)\leq \Delta_{X,\phi}(\varepsilon)+\eta \ ,\]
whence due to $\eta >0$ being arbitrary, we have that
\[\displaystyle \lim_{\varepsilon \rightarrow \varepsilon_{0_-}} \Delta_{X,\phi}(\varepsilon)=\Delta_{X,\phi}(\varepsilon_0) \ .\]
\end{proof}

\begin{lemma}\label{contAbove}
Let $\phi $ be a minimalizable measure of noncompactness and let $X$ be a Banach space with the Radon-Nikodym property. Then the modulus of noncompact convexity $\Delta_{X,\phi}(\varepsilon)$ is a function continuous from above on $[0,\phi(\overline{B}_X))$.
\end{lemma}
\begin{proof}
Let $\eta >0$ and $\varepsilon _0\in [0,\phi (\overline{B}_X))$. By definition \ref{modulusDefn} there exists a convex and closed subset $A\subset \overline{B}_X$, $\phi (A)> \varepsilon _0$, such that
\[1-d(0,A)<\Delta_{X,\phi}(\varepsilon _0)+\eta \ .\]
Since the measure of noncompactness  $\phi$ is minimalizable, for arbitrary $\gamma >0$, there exists an infinite $\phi$-minimal subset $B\subset A$, such that $\phi (B)\geq \phi (A)-\gamma$. Let $n_0\in \mathbb{N}$ be such that $\displaystyle \frac{\phi (A)-\gamma}{n_0}<\frac{diam B}{2}$. Using the Radon-Nikodym property of $X$ for the set $B$ and $\displaystyle r=\frac{\phi (A)-\gamma}{n_0}$, there exists $x_0\in B$  such that
\[x_0\notin \overline{co}\left (B\backslash \overline{B}(x_0,r)\right ) \ .\]
In this case the set $C=\overline{co}\left (B\backslash \overline{B}(x_0,r)\right )$ is a convex and closed subset of the set $B$ for which we have \[1-d(0,C)<\Delta_{X,\phi}(\varepsilon _0)+\eta \] and \[\phi (C)=\phi (B)\geq \phi (A)-\gamma \ .\]
Let $\displaystyle k=1+\frac{1-d(0,C)}{2}$. The set $A^*=kC\cap \overline{B}_X$ is a convex and closed set such that  \[1-d(0,A^*)<\Delta_{X,\phi}(\varepsilon _0)+\eta \ ,\]
and $\phi (A^*)=\phi (kB)=k\phi (B)\geq k(\phi (A)-\gamma)$. \\
Let $\displaystyle \varepsilon '\in \left ( \frac{2\varepsilon _0}{3-d(0,C)},\varepsilon _0\right )$ be arbitrary and let $\gamma =\varepsilon _0-\varepsilon '$. Using the above, there exists a subset $B\subset A$ such that
\[\phi (B)\geq \phi (A)-\varepsilon _0+\varepsilon '> \varepsilon ' \ .\]
Hence for the set $A^*$ we have that
\[\phi (A^*)=k\phi (B)> k\varepsilon ' \ .\]
If we choose $\delta =k\varepsilon '-\varepsilon _0$, then for an arbitrary $\varepsilon \in (\varepsilon_0, \varepsilon _0+\delta)$ we have that
\[\phi (A^*)> k\varepsilon '>\varepsilon \ .\]
Therefore
\[\inf \{ 1-d(0,A^*): A^*\subset \overline{B}_X, A^*=\overline{co}A^*, \phi (A^*)>\varepsilon \} \leq \Delta_{X,\phi}(\varepsilon_0)+\eta \ ,\] i.e.
\[\displaystyle \lim_{\varepsilon \rightarrow \varepsilon_{0}+} \Delta_{X,\phi}(\varepsilon)=\Delta_{X,\phi}(\varepsilon_0) \ .\]
\end{proof}

\section{Discussion}\label{discussionSec}
It is known that the separation measure of noncompactness $\beta$ is a minimalizable measure on a complete metric space. Therefore it is minimalizable on Banach spaces with the Radon-Nikodym property. Using Theorem \ref{mainThm} we can conclude that the modulus $\Delta_{X,\beta}(\varepsilon)$ with respect to the separation measure $\beta$ is a continuous function on $[0,\phi (\overline{B}_X))$ for a Banach space $X$ with the Radon-Nikodym property.
For example, since the spaces $l_p$ and $L_p$ ($1<p<+\infty$) are reflexive, they are with Radon-Nikodym property. Therefore  $\Delta_{l_p,\beta}(\varepsilon)$ and $\Delta_{L_p,\beta}(\varepsilon)$ are continuous functions on  $\left[0,\phi \left(\overline{B}_{l_p}\right)\right)$ and $\left[0,\phi \left(\overline{B}_{L_p}\right)\right)$ respectively.

The Kuratowski measure of noncompactness $\alpha$ is not minimalizable on spaces $l_p$  ($1<p<+\infty$). Indeed, if $\alpha$ was minimalizable on these spaces, then for $\varepsilon \in \left(0, 2\frac{\sqrt[n]{2}-1}{\sqrt[n]{2}}\right)$ ($n\in \nN$) and for $B_{l_p}$, there would exist an $\alpha$-minimal set $A\subset B_{l_p}$ such that $\alpha (A)\geq \alpha (B_{l_p})-\varepsilon=2-\varepsilon$. Since every $\alpha$-minimal set is also $\beta$-minimal and because $\alpha(A)=\beta (A)$ (see \cite{book2}, Lemma III. 2.9), we have that
$$2-\varepsilon\leq \beta(A) \leq \beta (B_{l_p})=2^{\frac{1}{p}} \ .$$
Because of the choice of $\varepsilon$ this would mean that $p\leq 1$, which is a contradiction with this choice of $p$.

Therefore, the measure of noncompactness $\alpha$ is strictly minimalizable on the space $l_1$. Namely, since $\beta (A) =2\chi(A)$ for an arbitrary bounded set $A \subset l_1$, (see \cite{book2}, Corollary X.4.7), and the general property
$$\chi(A)\leq \beta(A)\leq \alpha(A)\leq 2\chi(A) $$
holds, we conclude that $\alpha(A)=\beta(A)=2\chi(A)$. Since $l_1$ is weakly compactly determined, see \cite{article.7}, the Hausdorff measure of noncompactness $\chi$ is strictly minimalizable on $l_1$, which means that the measure $\alpha$ is also strictly minimalizable on $l_1$. Hence $l_1$ is the separable dual of the space $c_0$, so it is a space with the Radon-Nikodym property. Using Theorem \ref{mainThm} we conclude that $\Delta_{l_1,\alpha}(\varepsilon)$  is a continuous function on $[0,2)$.


\end{document}